\documentclass[a4paper,12pt]{article}

\usepackage{amsmath,amssymb,amsthm,graphicx,color}
\usepackage[colorlinks=true,linkcolor=blue,citecolor=blue,pdfpagelabels=false,pdfstartview=FitH]{hyperref}
\usepackage[margin=2.17cm,top=2.6cm,bottom=3cm]{geometry}
\numberwithin{equation}{section}

\def\for{\hskip0.9pt|\hskip0.9pt}

\outer\def\proclaime #1. #2\par{\medbreak
\noindent{\bf#1.\enspace}{\rm#2}\par
  \ifdim\lastskip<\medskipamount \removelastskip\penalty55\medskip\fi}
\def\bx{\bar x}  \def\bw{\bar w}
 
  \def\by{\bar y}  
   
\def\dom{\mathop{\rm dom}\nolimits} 
\def\gph{\mathop{\rm gph}\nolimits}

\def\argmin{\mathop{\rm argmin}}
\def\oball{{I\kern -.35em B}^{\circ}}
\def\reg F(\bx\for 0){\mathop{\rm reg}\nolimits}

\def\epsilon{\varepsilon}             \def\phi{\varphi}

\def\bx{\bar x} \def\by{\bar y} 

\def\tto{\rightrightarrows}
\def\dom{\mathop{\rm dom}\nolimits}

 \def\eps{\varepsilon}           \def\phi{\varphi}
\def\ball{{I\kern -.35em B}}

\newcommand {\del} {\partial}

\def\R{{\mathbb R}}

\newcommand{\Rex}{\R \cup \{\infty\}}

\newcount\nFran
\newcommand{\fran}{%
  \advance\nFran 1
  \ifodd \nFran
    \color{blue}
  \else
    \color{black}
  \fi}

\newcount\nMichel
\newcommand{\michel}{%
  \advance\nMichel 1
  \ifodd \nMichel
    \color{red}
  \else
    \color{black}
  \fi}

\newtheorem{definition}{Definition}[section]
\newtheorem{theorem}[definition]{Theorem}

\newtheorem{proposition}[definition]{Proposition}
\newtheorem{corollary}[definition]{Corollary}
\newtheorem{rem}[definition]{Remark}

\begin{document}

\title{Metric subregularity of the convex subdifferential in Banach spaces}

\author{Francisco J. Arag\'{o}n Artacho\thanks{Centre for Computer Assisted Research Mathematics and its Applications (CARMA), University of Newcastle, Callaghan, NSW 2308, Australia. E-mail:
\texttt{francisco.aragon@ua.es}. Partially supported by various Australian Research Council grants.},\;
Michel H.
Geoffroy\thanks{LAMIA, Dpt. de Math\'ematiques,
Universit\'e des Antilles et de la Guyane, F-97159 Pointe-\`a-Pitre,
Guadeloupe, {\tt michel.geoffroy@univ-ag.fr}. This author is
supported by Contract EA4540 (France).}}

\maketitle
\begin{center}
\noindent\emph{This paper is dedicated to Professor Simeon Reich on the occasion of his 65th birthday}
\end{center}

\smallskip

\begin{abstract}
In~\cite{AGE08} we characterized in terms of a quadratic growth condition various metric regularity properties of the subdifferential of a lower semicontinuous convex function acting in a Hilbert space. Motivated by some recent results in~\cite{MN12} where the authors extend to Banach spaces the characterization of the strong regularity, we extend as well the characterizations for the metric subregularity and the strong subregularity given in~\cite{AGE08} to Banach spaces. We also notice that at least one implication in these characterizations remains valid for the limiting subdifferential without assuming convexity of the function in Asplund spaces. Additionally, we show some direct implications of the characterizations for the convergence of the proximal point algorithm, and we provide some characterizations of the metric subregularity and calmness properties of solution maps to parametric generalized equations
\end{abstract}

\noindent {\bf Keywords:} Subdifferential, metric regularity, metric subregularity, strong subregularity, quadratic growth\\

\noindent {\bf AMS 2000 Subject Classification:} 49J52, 49J53.

~\bigbreak

\section{Introduction and preliminaries}\label{secint}
In~\cite{AGE08} we established a characterization of various metric regularity concepts for the
subdifferential of a proper lower semicontinuous convex function acting in a Hilbert space. More precisely, we considered
 the metric subregularity, the strong
metric subregularity, and the (strong) metric regularity of such an operator and we showed
that each of these properties is equivalent to a local quadratic growth condition on the function.

Recently, Drusvyatskiy and Lewis~\cite{DL13} have proved that the characterization for the strong metric regularity remains valid for the limiting subdifferential at $\bx$ for $0$ (where $\bx$ is a local minimimizer) of a not necessarily convex function in $\R^n$ when the function is subdifferentially continuous at $\bx$ for $0$ (see~\cite[Th.~3.3]{DL13}). Subsequently, Mordukhovich and Nghia~\cite{MN12} have generalized this latter result to Asplund spaces without the assumption of subdifferential continuity (see~\cite[Cor.~3.3]{MN12}), and have also extended our characterization of the strong regularity of the subdifferential of a convex function from Hilbert spaces to Banach spaces (see~\cite[Cor.~3.9]{AGE08} and~\cite[Th.~3.1]{MN12}). Motivated by this, we prove that the characterizations given in~\cite{AGE08} for the metric subregularity and the strong subregularity of the subdifferential of a convex function remain valid in Banach spaces. We also notice, see Remark~\ref{rem:1}\ref{rem:1iii}, that at least one implication in these characterizations remains valid for the (Mordukhovich) limiting subdifferential without assuming convexity of the function in Asplund spaces.

Throughout $X$ and $Y$ are real Banach spaces. We write $X^*$ for the real dual space of continuous linear
functionals. We denote a set-valued mapping from $X$ into the subsets of $Y$ by $F:X\tto Y$. The graph of $F$ is the set $\gph F=\{(x,y)\in X\times Y\mid y\in F(x)\}$, while $F^{-1}$ is the inverse mapping of $F$ defined by $x\in F^{-1}(y)\iff y\in F(x)$. Single-valued mappings, also called functions, are represented by $f:X\to Y$.

A function $f:X\to \R\cup\{\infty\}$ is said to be convex if
$$f\big((1-\lambda) x+\lambda y\big)\leq (1-\lambda)f(x)+\lambda f(y),$$
 for all $\lambda\in(0,1)$ and $x,y\in\dom f=\{z\in X\mid f(z)<\infty\}$.
The (convex) subdifferential of a (not necessarily convex) function $f:X\to\Rex$ is the set-valued mapping $\partial f:X\tto X^*$ which is defined at any point $\bx\in\dom f$ by
$$\partial f(\bx):=\{y^*\in X^*\mid f(x)\geq f(\bx)+\langle y^*,x-\bx\rangle\quad \text{for all } x\in X\}.$$
Therefore, $\by^*\in\partial f(\bx)$ if and only if $\bx$ is a global minimizer of the tilted function $f(\cdot)-\langle \by^*,\cdot\rangle$.
A function is said to be proper if its domain is nonempty. We will denote by $\Gamma(X)$
the space of all proper lower semicontinuous convex functions from $X$ into $\Rex$.

The closed unit ball is denoted by $\ball$, while $\ball_a(x)$ stands for the closed ball of radius $a$ centered at $x$. We denote by $d(x,C)$ the distance from a point $x$ to a set $C$, i.e., $d(x,C)=\inf_{y\in C}\|x-y\|$. If $C$ is empty, we adopt the convention $d(x,C)=\infty$ for any $x\in X$. For sets $C$ and $D$ in $X$, the \emph{excess} of $C$ beyond $D$ is defined by $e(C,D)=\sup_{x\in C}d(x,D)$, with the convention $e(\emptyset,D)=0$ when $D\neq\emptyset$, and $e(\emptyset,\emptyset)=\infty$.

Our study is focused on two key notions: \emph{metric subregularity} and \emph{strong metric subregularity}. They are defined as follows.

\begin{definition}
A mapping $F:X\tto Y$ is said to be \emph{metrically subregular}
at $\bx$ for $\by$ if  $\by\in F(\bx)$ and there is a positive constant $\kappa$ along with neighborhoods $U$ of $\bx$ and $V$
of $\by$ such such that
\begin{equation} \label{subreg2}
d(x,F^{-1}(\by))\leq\kappa d(\by,F(x)\cap V), \text{ for all } x \in U.
\end{equation}
\end{definition}

\begin{definition}
A mapping $F:X\tto Y$ is said to be \emph{strongly (metrically) subregular}
at $\bx$ for $\by$ if  $\by\in F(\bx)$ and there is a positive constant $\kappa$ along with neighborhoods $U$ of $\bx$ and $V$
of $\by$ such that
\begin{equation} \label{strongsubreg2}
\|x-\bx\|\leq\kappa d(\by,F(x)\cap V), \text{ for all } x \in U.
\end{equation}
\end{definition}
Equivalently, $F$ is strongly metrically subregular at $\bx$ for $\by$ if it is metrically subregular
at $\bx$ for $\by$ and, in addition, $\bx$ is an isolated point of $F^{-1}(\by)$.
The definition of metric subregularity can be simplified in the following way:
\begin{equation} \label{subreg}
d(x,F^{-1}(\by))\leq\kappa d(\by,F(x)), \text{ for all } x \in U',
\end{equation}
for a possibly smaller neighborhood $U'$ of $\bx$, see~\cite[pp.~371--372]{AGE08} for details.
Likewise, the definition of strong subregularity can be simplified as
\begin{equation} \label{strongsubreg}
\|x-\bx\|\leq\kappa d(\by,F(x)), \text{ for all } x \in U',
\end{equation}
for a possibly smaller neighborhood $U'$ of $\bx$.

Relation~\eqref{strongsubreg} implies in particular that $\bx$ is an isolated point of $F^{-1}(\by)$.
If a linear mapping $A$ is strongly metrically subregular at $\bx$ for $\by=A\bx$ then  $A^{-1}\by=\{\bx\}$  and $A$ is injective; in finite dimensional spaces it is an equivalence.
More generally, a polyhedral (set-valued) mapping $F$, i.e., a mapping whose graph is the union of finitely many polyhedral convex sets, is
metrically subregular with the same constant at any point of its graph, and it is strongly subregular at $\bx$ for $\by$ if and only if $\bx$ is an isolated point of $F^{-1}(\by)$. This is a consequence of~\cite[Prop.~1]{R81}, see also~\cite{DRO09} for more information about these properties. 

Our main tool will be the well-known Ekeland's variational principle, see~\cite[Th.~1.1]{E74}.

\begin{theorem}[Ekeland's variational principle]\label{th:Ekeland}
Let $(X,d)$ be a complete metric space, and $f:X\to\Rex$ a proper lower semicontinuous function bounded from below. Suppose that for some $u\in X$ and some $\eps>0$, $f(u)\leq \inf_{x\in X} f(x) +\eps$. Then for every $\lambda>0$ there exists some point $v\in X$ such that
\begin{gather}
d(u,v)\leq \lambda,\\
f(v)+(\eps/\lambda)d(u,v)\leq f(u),\\
f(x)>f(v)-(\eps/\lambda)d(v,x),\quad\forall x\neq v.
\end{gather}
\end{theorem}
\smallskip

The rest of the paper is organized as follows. In Section~2 we prove that the characterization of the metric subregularity given in~\cite{AGE08} for Hilbert spaces remains valid in Banach spaces. In Section~3 we extend as well the characterization of the strong subregularity to Banach spaces, and show some additional characterizations of this property. The last Section~4 contains some direct consequences of the main results proved in the paper regarding the convergence of the proximal point algorithm, and we show some characterizations of the metric subregularity and calmness properties of solution maps to parametric generalized equations. These consequences are actually our main motivation for studying the metric regularity properties of the subdifferential, and the reason why we believe in the importance in characterizing these properties.

\section{Characterization of metric subregularity}

We begin by showing a characterization of the metric subregularity of the subdifferential of a proper lower semicontinuous function. This result, originally proved in~\cite{AGE08} in Hilbert spaces, remains valid in Banach spaces with some adjustments in the proof, which has also been simplified.

\begin{theorem}[{\cite[Theorem~3.3]{AGE08}}]\label{th_subreg}
  Given a Banach space $X$, consider a function $f\in\Gamma(X)$
  and points $\bx\in X$ and $\by^*\in X^*$ such that $\by^*\in\partial f(\bx)$.
  Then $\partial f$ is metrically subregular at $\bx$ for
  $\by^*$ if and only if there exist a neighborhood $U$ of $\bx$ and a
  positive constant $c$ such that
  \begin{equation}\label{charac_subreg}
    f(x)\geq f(\bx)+\langle \by^*,x-\bx\rangle +cd^2(x,(\partial f)^{-1}(\by^*)) \text{ whenever } x\in U.
  \end{equation}
  Specifically, if $\partial f$ is metrically subregular at $\bx$ for $\by^*$ with constant $\kappa$, then~\eqref{charac_subreg} holds for all  $c<1/(4\kappa)$; conversely,
  if~\eqref{charac_subreg} holds with constant $c$, then $\partial f$ is metrically subregular at $\bx$ for $\by^*$ with constant $1/c$.
\end{theorem}

\begin{proof}
  Assume first that~\eqref{charac_subreg} holds. Fix $x\in U$ and consider any $y^*\in\partial f(x)$ (if $\partial f(x)=\emptyset$ there is nothing to prove). Choose any $\eps>0$. Since $(\partial f)^{-1}(\by^*)\neq\emptyset$, there is some $x_\eps\in(\partial f)^{-1}(\by^*)$ such that $\|x-x_\eps\|\leq d(x,(\partial f)^{-1}(\by^*))+\eps$. Then, by definition of the subdifferential,
  \begin{gather*}
  \langle y^*,x-x_\eps\rangle \geq f(x)-f(x_\eps),\\
  -\langle \by^*,\bx-x_\eps\rangle \geq f(x_\eps)-f(\bx).
  \end{gather*}
  Moreover, by~\eqref{charac_subreg},
  $$-\langle \bar y^*,x-\bx\rangle \geq f(\bx)-f(x)+cd^2(x,(\partial f)^{-1}(\by^*));$$
  whence,
  \begin{align*}
  \|y^*-\by^*\|(d(x,(\partial f)^{-1}(\by^*))+\eps)&\geq \|y^*-\by^*\|\|x-x_\eps\|\\
  &\geq\langle y^*-\by^*,x-x_\eps\rangle\\
  &=\langle y^*,x-x_\eps\rangle-\langle\by^*,x-\bx\rangle-\langle\by^*,\bx-x_\eps\rangle\\
  &\geq cd^2(x,(\partial f)^{-1}(\by^*)).
  \end{align*}
  Thus, taking limits when $\eps$ goes to zero,
  $$cd^2(x,(\partial f)^{-1}(\by^*))\leq \|y^*-\by^*\|d(x,(\partial f)^{-1}(\by^*)).$$
  If $d(x,(\partial f)^{-1}(\by^*))=0$ then $x\in (\partial f)^{-1}(\by^*)$, since the set $(\partial f)^{-1}(\by^*)$ is closed, whence $d(\by^*,\partial f(x))=0$ and we are done. Otherwise,
  $$d(x,(\partial f)^{-1}(\by^*))\leq \frac{1}{c}\|y^*-\by^*\|,$$
  and since $y^*\in\partial f(x)$ was arbitrarily chosen, we obtain
  $$d(x,(\partial f)^{-1}(\by^*))\leq\frac{1}{c}\, d(\by^*,\partial f(x)),\quad\text{for all }x\in U;$$
  that is, $\partial f$ is metrically subregular at $\bx$  for $\by^*$ with  constant $1/c$.

  Conversely, if $\partial f$ is metrically subregular at $\bx$ for $\by^*$ with constant $\kappa$, there is some positive constant $a$ such that
  \begin{equation}\label{eq:strong_subreg}
  d(x,(\partial f)^{-1}(\by^*))\leq\kappa d(\by^*,\partial f(x))\quad\text{for all } x\in \ball_a(\bx).
  \end{equation}
  We will prove by contradiction that~\eqref{charac_subreg} holds for all $c<1/(4\kappa)$ and $U\subset\ball_{2a/3}(\bx)$.
  Otherwise, there is some $z\in\ball_{2a/3}(\bx)$  such that
  \begin{equation}\label{eq:contradic}
  f(z)+\langle\by^*,\bx-z\rangle<f(\bx)+cd^2(z,(\partial f)^{-1}(\by^*)).
  \end{equation}
  Observe that $\bx$ is a global minimizer of the lower semicontinuous convex function $f(\cdot)+\langle\by^*,\bx-\,\cdot\,\rangle$ since $\by^*\in\partial f(\bx)$. Additionally,   \eqref{eq:contradic} and $\by^*\in\partial f(\bx)$ implies $d(z,(\partial f)^{-1}(\by^*))>0$.  By Ekeland's variational principle of Theorem~\ref{th:Ekeland}, there exists some $u\in X$ such that $\|u-z\|\leq\frac{1}{2}d(z,(\partial f)^{-1}(\by^*))$ and for all $x\in X$,
  \begin{align*}
f(x)+\langle\by^*,\bx-x\rangle &\geq f(u)+\langle\by^*,\bx-u\rangle -\frac{cd^2(z,(\partial f)^{-1}(\by^*))}{\frac{1}{2}d(z,(\partial f)^{-1}(\by^*))}\|x-u\|\\
  &=f(u)+\langle\by^*,\bx-u\rangle -2cd(z,(\partial f)^{-1}(\by^*))\|x-u\|.
  \end{align*}
  Hence, $u$ minimizes the convex function $f(\cdot)+\langle\by^*,\bx-\,\cdot\,\rangle+2cd(z,(\partial f)^{-1}(\by^*))\|\,\cdot\,-u\|$; whence,
  \begin{align}
  0&\in \partial \big(f(\cdot)+\langle\by^*,\bx-\,\cdot\,\rangle+2cd(z,(\partial f)^{-1}(\by^*))\|\,\cdot\,-u\|\big)(u)\nonumber\\
  &=\partial f(u)-\by^*+2cd(z,(\partial f)^{-1}(\by^*))\ball\label{eq:sum_rule},
  \end{align}
  where we have used in the equality the subdifferential sum rule (see, e.g.,~\cite[Th.~4.1.19]{BV_CF}).
  Therefore, there is some $y^*\in\partial f(u)$ such that $\|y^*-\by^*\|\leq 2cd(z,(\partial f)^{-1}(\by^*))$. Additionally, since
  $$d(z,(\partial f)^{-1}(\by^*))\leq\|z-u\|+d(u,(\partial f)^{-1}(\by^*))\leq \frac{1}{2}d(z,(\partial f)^{-1}(\by^*))+d(u,(\partial f)^{-1}(\by^*)),$$
  one has $0<d(z,(\partial f)^{-1}(\by^*))\leq 2d(u,(\partial f)^{-1}(\by^*))$, and thus,
  $$d(\by^*,\partial f(u))\leq\|y^*-\by^*\|\leq 4cd(u,(\partial f)^{-1}(\by^*))<\frac{1}{\kappa}d(u,(\partial f)^{-1}(\by^*)).$$
  This strict inequality contradicts~\eqref{eq:strong_subreg}, since
  $$\|u-\bx\|\leq \|u-z\|+\|z-\bx\|\leq\frac{3}{2}\|z-\bx\|\leq a,$$
  which completes the proof.
\end{proof}

\begin{rem}~\label{rem:1}
\begin{enumerate}
\item We believe that it may be possible to improve the bound $c<1/(4\kappa)$ in Theorem~\ref{th_subreg}. Nevertheless, observe that this bound seems to be rather tight, since~\eqref{charac_subreg} might be false for $c=1/\kappa$, as it happens for the real function $f(x)=x^2$.
\item Notice that the first part of the proof holds without convexity:~\eqref{charac_subreg} implies metric subregularity of the (convex) subdifferential.
\item \label{rem:1iii}The second part of the proof remains valid in Asplund spaces without assuming convexity of $f$ for the (Mordukhovich) limiting subdifferential if in addition $\bx$ is a local minimizer of the function $f(\cdot)+\langle\by^*,\bx-\cdot\,\rangle$. In this case the limiting subdifferential sum rule (see~\cite[Cor.~4.3]{MS96} or~\cite[Th.~3.6]{M06}) gives us an inclusion in~\eqref{eq:sum_rule}.
\end{enumerate}
\end{rem}

\section{Characterization of strong metric subregularity}

The following characterization for the strong subregularity of the subdifferential in Banach spaces can be easily derived as a consequence of the one for the metric subregularity in Theorem~\ref{th_subreg}.

\begin{theorem}[{\cite[Theorem~3.5]{AGE08}}]\label{th_strong_subreg}
  Given a Banach space $X$, consider a function $f\in\Gamma(X)$
  and points $\bx\in X$ and $\by^*\in X^*$ such that $\by^*\in\partial f(\bx)$.
  Then $\partial f$ is strongly subregular at $\bx$ for
  $\by^*$ if and only if there exist a neighborhood $U$ of $\bx$ and a
  positive constant $c$ such that
  \begin{equation}\label{charac_strong_subreg}
    f(x)\geq f(\bx)+\langle \by^*,x-\bx\rangle +c\|x-\bx\|^2 \text{ whenever } x\in U.
  \end{equation}
  Specifically, if $\partial f$ is strongly subregular at $\bx$ for $\by^*$ with constant $\kappa$, then~\eqref{charac_strong_subreg} holds for all  $c<1/(4\kappa)$; conversely,
  if~\eqref{charac_strong_subreg} holds with constant $c$, then $\partial f$ is strongly subregular at $\bx$ for $\by^*$ with constant $1/c$.
\end{theorem}

\begin{proof}
  Assume first that~\eqref{charac_strong_subreg} holds. Let $x\in U$ be such that $\by^*\in\partial f(x)$. Then
  $$\langle \by^*,x-\bx\rangle \geq f(x)-f(\bx),$$
  and~\eqref{charac_strong_subreg} implies $x=\bx$. Therefore,
  $(\partial f)^{-1}(\by)\cap U=\{\bx\}$. In addition,~\eqref{charac_strong_subreg} implies~\eqref{charac_subreg}, and hence Theorem~\eqref{th_subreg} implies that $\partial f$ is strongly subregular at $\bx$ for $\by^*$ with constant~$1/c$.

  Conversely, if $\partial f$ is strongly subregular at $\bx$ for $\by^*$ with constant $\kappa$, then there is some neighborhood $U$ of $\bx$ such that~\eqref{charac_subreg} holds for all  $c<1/(4\kappa)$ and $(\partial f)^{-1}(\by^*)\cap U=\{\bx\}$. We may assume without loss of generality that $U=\ball_{2a}(\bx)$, for some positive constant $a$. Pick any $x\in\ball_{a}(\bx)$ and let $z\in(\partial f)^{-1}(\by^*)$. If $z\not\in\ball_{2a}(\bx)$, then
  $$\|x-z\|\geq\|z-\bx\|-\|x-\bx\|\geq 2a-a=a\geq\|x-\bx\|;$$
  whence,
  $$d(x,(\partial f)^{-1}(\by^*))=d(x,(\partial f)^{-1}(\by^*)\cap\ball_{2a}(\bx))=\|x-\bx\|.$$
  Then~\eqref{charac_strong_subreg} holds for any $x\in\ball_a(\bx)$.
\end{proof}

\begin{rem}
Similar observations to the ones in Remark~\ref{rem:1} apply to Theorem~\ref{th_strong_subreg}.
\end{rem}

A point $\bx$ is a global minimizer of the function $f$ if and only if $0\in\partial f(\bx)$. It follows from Theorem~\ref{th_strong_subreg}  that
 $\del f$ is strongly metrically subregular at $\bx$ for $0$ if and
only if the mapping $f$ satisfies the following {\em quadratic growth condition}:
$$ f(x)\geq \inf f +c\|x-\bx\|^2\; \text{ for all } x \text{ close to }\bx.$$
Particular forms of the above inequality are known to be equivalent to the second order sufficient condition in
nonlinear programming problems with qualified constraints (see, e.g.,~\cite{BIO95a}).


The following result regarding the strong subregularity of the subdifferential of the sum of two functions is a straightforward consequence of the characterization in Theorem~\ref{th_strong_subreg}.

\begin{corollary}[Sum of strongly subregular subdifferentials]
Let $f,g \in \Gamma(X)$ and let $\bx\in X$ and $\by^*, \bw^* \in X^*$ such that $\by^* \in \partial f(\bx)$ and $ \bw^* \in \partial g(\bx)$.  If  $\partial f$ and $\partial g$ are respectively strongly subregular at $\bx$ for $\by^*$ and $\bx$ for $\bw^*$ then $\partial (f+g)$ is strongly subregular at $\bx$ for $\by^*+\bw^*$.
\end{corollary}
\begin{proof}
By Theorem~\ref{th_strong_subreg}, the strong subregularity of $\partial f$ (respectively $\partial g$)
 at $\bx$ for $\by^*$ (respectively at $\bx$ for $\bw^*$)
yields (and is actually equivalent to) the existence of two neigborhoods $U_1$ and $U_2$ of $\bx$
and two positive constants $c_1$ and $c_2$ such that
\begin{equation}
f(x) \geq f(\bx) +\langle \by^*, x-\bx\rangle + c_1\|x-\bx\|^2,\; \forall x \in U_1,
\end{equation}
and
\begin{equation}
g(x) \geq g(\bx) +\langle \bw^*, x-\bx\rangle + c_2\|x-\bx\|^2,\; \forall x \in U_2.
\end{equation}
By adding to each other the above inequalities we get
\begin{equation}
(f+g)(x) \geq (f+g)(\bx) +\langle \by^*+\bw^*, x-\bx\rangle + (c_1+c_2)\|x-\bx\|^2,\; \forall x \in U,
\end{equation}
where $U:=U_1\cap U_2$; that is, $\partial (f+g)$ is strongly subregular at $\bx$ for $\by^*+\bw^*$,
which completes the proof.
\end{proof}

It turns out that strong subregularity of the subdifferential is also related to local strong monotonicity, which is defined as follows.

\begin{definition}\label{def1}
Given a mapping $T:X\tto X^*$, the point $(\bar x,\bar y^*)\in\gph T$  is said to be \emph{(locally) strongly monotonically related} to $\gph T$
 if there are some neighborhoods $U$ of $\bx$ and $V$ of $\by^*$ together
with some positive constant $c$ such that
$$\langle y^*-\by^*,x-\bx\rangle\geq c\|x-\bx\|^2,\quad\text{for all } (x,y^*)\in\gph T\cap(U\times V).$$
\end{definition}

When $X$ is a Hilbert space, it is easy to check that a point $(\bx,\by)\in\gph T$ is locally strongly monotonically related to $\gph T$ with constant $c$ if and only if the point $(\bx,\by-c\bx)\in\gph (T-cI)$ is
\emph{locally monotonically related} to $\gph (T-cI)$, where $I$ stands for the identity mapping; that is,
$$\langle z-(\by-c\bx),x-\bx\rangle\geq 0,\quad\text{for all } (x,y)\in\gph (T-cI)\cap(U\times V).$$

\begin{rem}
Recall that a mapping $T:X\tto X^*$ is strongly monotone if there exists $c>0$ such that $\langle y_2^*-y_1^*,x_2-x_1\rangle \geq c \|x_2-x_1\|^2$ whenever $y_1^*\in T(x_1), y_2^*\in T(x_2).$ Consequently, if a mapping $T$ is  strongly monotone then any point $(\bar x,\bar y^*)\in\gph T$ is locally strongly monotonically related to $\gph T$.
\end{rem}

\begin{theorem}[Characterization of strong metric subregularity of subdifferentials]\label{thcharstrgsubreg}
Consider  a function $f$ in  $\Gamma(X)$ and points $\bx\in X$ and
$\by^*\in X^*$ such that $\by^*\in\partial f(\bx)$. The following assertions are equivalent.\medbreak
\noindent {\rm (i)}  The set-valued mapping $\partial f$ is strongly subregular at $\bx$ for  $\by$.\smallbreak
\noindent {\rm (ii)} There exist a neighborhood $U$ of $\bx$ and a
  positive constant $c$ such that
  \begin{equation}\label{charac_strong_subreg_bis}
    f(x)\geq f(\bx)+\langle \by^*,x-\bx\rangle +c\|x-\bx\|^2 \text{ whenever } x\in U.
  \end{equation}
{\rm (iii)}  There exist  a neighborhood $U$ of $\bx$ and a positive constant $c$  such that
\begin{equation}\label{eq:str_mon}
\langle y^*-\by^*,x-\bx\rangle \geq  c\|x-\bx\|^2,\quad \forall x \in U, y^*\in \partial f(x).
\end{equation}
{\rm (iv)}  The point $(\bx,\by)$ is locally strongly monotonically related to $\gph(\partial f)$.

\end{theorem}
\begin{proof}
{\bf (i) $\Leftrightarrow$ (ii)} has been proved in Theorem~\ref{th_strong_subreg}.\\
{\bf (ii) $\Rightarrow$ (iii).} Suppose that (ii) holds and pick any $x \in U$ and $y^* \in \partial f(x)$. Then $f(\bx) \geq f(x) + \langle y^*,\bx-x \rangle$, and thus
$$\langle y^*-\by^*,x-\bx \rangle= \langle y^*,x-\bx \rangle +\langle \by^*,\bx-x \rangle \geq c\|x-\bx\|^2,$$
and (iii) holds.\\
{\bf (iii) $\Rightarrow$ (iv)} is straightforward.\\
{\bf  (iv) $\Rightarrow$ (i).} Let $c$, $U$ and $V$ as in Definition~\ref{def1}. Then
$$\langle y^*-\by^*,x-\bx\rangle \geq  c\|x-\bx\|^2,\quad \text{for all } (x, y^*)\in \gph(\partial f)\cap(U\times V).$$
Take $x \in U\backslash\{\bx\}$. If $\partial f(x)\cap V=\emptyset$ we are done. Otherwise, pick any $y^* \in \partial f(x)\cap V$. Then
$$\|x-\bx\|^2\leq \displaystyle\frac{1}{c} \langle y^*-\by^*,x-\bx\rangle \leq \frac{1}{c} \|y^*-\by^*\|\|x-\bx\|.$$
Hence, $\|x-\bx\|\leq \displaystyle\frac{1}{c}  \|y^*-\by^*\|$, and being this valid for all $y^* \in \partial f(x)\cap V$, we obtain
$$\|x-\bx\|\leq\displaystyle\frac{1}{c}d(\by^*,\partial f(x)\cap V),$$
i.e., $\partial f$ is strongly subregular at $\bx$ for $\by^*$ with constant $\displaystyle\frac{1}{c}$.
\end{proof}

The {\em contingent derivative}, a graphical concept of derivative for set-valued maps, was initiated by Aubin in~\cite{AUB81};
its definition strongly relies on the notion of contingent cone (independently introduced by Bouligand and Severy in 1930, see comments in~\cite[p. 133]{M06}).  Recall that if $K\subset X$ and $x \in \overline{K}$ ($\overline{K}$ denoting the closure of $K$) then the \emph{contingent cone} $T_K(x)$ is defined by
$$T_K(x):=\left\{v \in X \,\bigg|\, \displaystyle \liminf_{\tau \to 0^+} \frac{d(x+v\tau,K)}{\tau}=0\right\}=\limsup_{\tau\downarrow 0}\frac{K-x}{\tau}.$$

\noindent The contingent derivative of  $F:X\tto Y$ at $(\bx,\by)\in\gph F$ is the set-valued map $DF(\bx\for\by)$ from $X$ to $Y$ defined by
$$DF(\bx\for\by)(w):=\{z\in Y\mid(w,z)\in T_{\gph F}(\bx,\by)\}.$$
Note that $D\partial f(\bx\for\nabla f(\bx))(w)=\{\nabla^2f(\bx)w\}$ when $f$ is twice (Fr\'echet) differentiable (see, e.g.,~\cite[Proposition~5.1.2]{AF90}). 
For more details on the contingent derivative, one can refer to the comprehensive monograph~\cite{AF90} by Aubin and Frankowska.

\begin{corollary}\label{CorContingent}
  Consider a function $f$ in $\Gamma(X)$. Then $\partial f$ is strongly subregular at $\bx$ for $\by^*$ if there is a constant $c>0$ such that $D\partial f(\bx\for\by^*)$ is \emph{positive-definite with modulus} $c$ in the sense that
  \begin{equation}\label{locmoncont}
    \langle z^*,w\rangle\geq c\|w\|^2,\text{ for all }w \in X\text{ and }z^*\in D\partial f(\bx\for\by^*)(w).
  \end{equation}
  Moreover, the converse also holds true when $\dim X<\infty$. Specifically, if~\eqref{locmoncont} holds then $\partial f$ is strongly subregular at $\bx$ for $\by$ for any constant $\kappa>1/c$.
  \end{corollary}

\begin{proof}
  Suppose first that $\partial f$ is strongly subregular at $\bx$ for $\by^*$. Let $w\in X$ and consider $z^*\in D\partial f(\bx,\by^*)(w)$. Then there are $(x_n,y_n^*)\in\gph\partial f$ and $\tau_n\downarrow 0$ with $(x_n,y_n^*)\to(\bx,\by^*)$ and $[(x_n,y_n^*)-(\bx,\by^*)]/\tau_n\to (w,z^*)$. Theorem~\ref{thcharstrgsubreg} implies the existence of a neighborhood $U$ and a constant $c$ such that~\eqref{eq:str_mon} holds. Since $x_n\in U$ eventually and $y_n^*\in\partial f(x_n)$, one has
  $$\langle y_n^*-\by^*,x_n-\bx\rangle \geq c\|x_n-\bx\|^2\quad \text{eventually.}$$
  Hence,
  $$\left\langle \frac{y_n^*-\by^*}{\tau_n},\frac{x_n-\bx}{\tau_n}\right\rangle \geq c\left\|\frac{x_n-\bx}{\tau_n}\right\|^2, \text{ eventually.}$$
  Making $n\to\infty$ we obtain $\langle z,w\rangle\geq c\|w\|^2.$

  We offer two proofs of the converse. For the first one, observe that~\eqref{locmoncont} implies in particular that $D\partial f(\bx\for\by^*)^{-1}(0)=\{0\}$, and then $\partial f$ is strongly subregular at $\bx$ for $\by^*$ by~\cite[Theorem~5.3]{DR04}.

  For the second proof of the converse, choose any $\kappa>1/c$. We are going to prove by contradiction that there are some neighborhoods $U$ of $\bx$ and $V$ of $\by^*$ such that
  \begin{equation}\label{cor_strsub}
    \|x-\bx\|\leq\kappa\|y^*-\by^*\|\text{ whenever }(x,y^*)\in(\gph\partial f)\cap (U\times V).
  \end{equation}
  Otherwise, for all $n\in\mathbb{N}$ there is $(x_n,y^*_n)\in\gph\partial f$ with $\|x_n-\bx\|\leq 1/n$, $\|y^*_n-\by^*\|\leq 1/n$ and such that $\|x_n-\bx\|>\kappa\|y^*_n-\by^*\|$. This implies $x_n\neq\bx$ for all $n$, and because of the finite dimensionality of the space $X$, the bounded sequences $(x_n-\bx)/\|x_n-\bx\|$ and $(y^*_n-\by^*)/\|x_n-\bx\|$ must have some convergent subsequences. Replacing the original sequences by these subsequences we may assume that there are some points $w$ and $z^*$ with
  $$\left(\frac{x_n-\bx}{\|x_n-\bx\|},\frac{y^*_n-\by^*}{\|x_n-\bx\|}\right)\to (w,z^*)\in\gph D\partial f(\bx\for\by^*).$$
  By~\eqref{locmoncont} we get
  $$1=\|w\|\geq\kappa\|z^*\|\geq\kappa c\|w\|>1,$$
  which is a contradiction. Condition~\eqref{cor_strsub} implies strong subregularity of $\partial f$ at $\bx$ for $\by^*$ with any constant $\kappa>1/c$.
\end{proof}

\begin{corollary}
  Consider a function $f$ in $\Gamma(X)$, with $\dim X<\infty$, and a point $\bx$ in $X$ such that $f$ is twice (Fr\'echet) differentiable in a neighborhood of $\bx$. Then $\nabla f$ is strongly subregular at $\bx$ for $\nabla f(\bx)$ if and only if there is a positive constant $c$ such that $\nabla^2f(\bx)$ is positive-definite with modulus $c$, that is,
  \begin{equation}\label{locmondif}
    \langle \nabla^2f(\bx)u,u\rangle \geq  c\|u\|^2,\quad \forall u \in X.
  \end{equation}
\end{corollary}

\begin{proof} 
  Apply Corollary~\ref{CorContingent}.
\end{proof}

%

The next result provides a sufficient condition for strong subregularity at any point in the subdifferential.

\begin{proposition}
Let $f$ be a function in $\Gamma(X)$. Assume that there is a neighborhood $U$ of some point $\bx \in X$ such that for all $x \in U,\, \lambda \in (0,1)$,
\begin{equation}\label{strgbx}
f((1-\lambda)x+\lambda \bx) \leq (1-\lambda)f(x)+\lambda f(\bx)-c\lambda(1-\lambda)\|x-\bx\|^2.
\end{equation}
Then for all $y^* \in \partial f (\bx), x \in U$,
\begin{equation}\label{stsub}
f(x)\geq f(\bx)+\langle y^*,x-\bx\rangle +c\|x-\bx\|^2.
\end{equation}
In particular, $\partial f$ is strongly subregular at $\bx$ for any point $y^*$ such that
$y^* \in \partial f(\bx)$.
\end{proposition}

\begin{proof}
Let $y^* \in \partial f(\bx),  \lambda \in (0,1)$ and $x \in U$. From~\eqref{strgbx} we get
$$
  (1-\lambda) f(x)  \geq f((1-\lambda)x+\lambda\bx)-\lambda f(\bx)+c\lambda(1-\lambda)\|x-\bx\|^2,$$
i.e., \begin{equation}\label{rel0}
  f(x)  \geq \frac{1}{1-\lambda} f((1-\lambda)x+\lambda\bx)- \frac{\lambda}{1-\lambda}f(\bx) +
  c\lambda\|x-\bx\|^2.\end{equation}
Moreover, since $y^* \in \partial f(\bx)$, one has $f((1-\lambda)x+\lambda\bx)\geq f(\bx)+(1-\lambda)\langle y^*,x-\bx\rangle$ and relation~\eqref{rel0} yields
$$
f(x) \geq f(\bx) + \langle y^*,x-\bx\rangle +c\lambda\|x-\bx\|^2.
$$
Making $\lambda \uparrow 1$ in the latter inequality, one obtains~\eqref{stsub}. 
\end{proof}

\begin{rem}
Observe that condition~\eqref{strgbx} is weaker than assuming \emph{strong convexity} of $f$ on $U$, which entails the existence of a constant $c>0$ such that
\begin{equation}\label{sconvdef}
f((1-\lambda)x_1+\lambda x_2)\leq (1-\lambda)f(x_1)+\lambda f(x_2) -c\lambda(1-\lambda)\|x_1-x_2\|^2,
\end{equation}
for all $x_1, x_2 \in U$ and $\lambda \in (0,1)$.
\end{rem}

\section{Consequences}

In this final section we will show some direct consequences of the characterizations of the metric subregularity and the strong subregularity of the subdifferential given in the previous sections.

\subsection{The proximal point algorithm}
The proximal point algorithm was developed by Rockafellar in~\cite{R76} for finding zeroes of maximally monotone operators. Rockafellar proved, in particular, that
the iterative process
\begin{equation}\label{PPM_old}
0\in \lambda_n(x_{n+1}-x_n)+T(x_{n+1})\; \text{ for } n=0,1,2,\ldots
\end{equation}
known as the \emph{exact} proximal point method (where $\lambda_n$
is a sequence of positive numbers and $x_0 \in X$ is the initial
point), provides a sequence $x_n$ which is weakly convergent to a
solution to the inclusion $0\in T(x)$ when $T$ is a maximally monotone operator.
The particular case when $T$ is the subdifferential of a lower
semicontinuous convex function is of special relevance; here the
subproblem~\eqref{PPM_old} becomes \begin{equation}\label{PPM_convex}
x_{n+1}:=\argmin_z
\left\{f(z)+\frac{\lambda_n}{2}\|z-x_n\|^2\right\}, \end{equation}
transforming thus the single problem of minimizing a convex
function into solving a sequence of problems where the objective
function is strongly convex, which improves the convergence
properties of some minimization algorithms (needed in order to solve~\eqref{PPM_convex}). In addition,
the term $\|z-x_n\|^2$ forces the next iteration to remain
\emph{proximal} to the previous one, while the parameter
$\lambda_n$ provides control on this effect.

In~\cite{ADG07} the authors propose a
generalization of the proximal point method without assuming monotonicity of the operator in their convergence results. This generalization basically
consists in replacing the constants $\lambda_n$ in~\eqref{PPM_old} by some functions
$g_n$ which are Lipschitz continuous on some neighborhood of $0$
with Lipschitz constants $\lambda_n$. This modification of the
method allows the mapping $T$ to act between two different (Banach) spaces
$X$ and $Y$.  More specifically,  choose a sequence of Lipschitz continuous
function $g_n : X \to Y$ and consider the following algorithm:
\begin{equation}\label{PPM}
0 \in g_n(x_{n+1}-x_n) +T(x_{n+1})\; \text{ for } n=0,1,2,\ldots.
\end{equation}
In particular, if $T$ is strongly subregular around a solution $\bx$ for $0$ with constant $\kappa>0$ and the Lipschitz constants $\lambda_n$ are  upper bounded by $1/(2\kappa)$, then any sequence satisfying~\eqref{PPM} and whose elements are sufficiently close
to $\bx$, is linearly convergent to this solution
(see~\cite[Theorem~4.2]{ADG07}). Furthermore, the convergence is
superlinear when $\lambda_n$ converges to $0$. When $T$ happens to
be strongly regular, the sequence exists and is unique (within a
neighborhood of the solution). Therefore, when $T=\partial f$ for $f\in\Gamma(X)$, the algorithm~\eqref{PPM} is (super)linearly convergent if the quadratic growth condition~\eqref{charac_strong_subreg} is satisfied. One can find a
similar condition to~\eqref{charac_strong_subreg} for the convergence of the classical algorithm~\eqref{PPM_convex} in~\cite[Th.~3.1]{HZ08}.

Another interesting approach can be found in~\cite{LEV09}, where
the author assumes both maximal monotonicity and metric
subregularity of the mapping
$T$ around some solution $\bx$, and proves the local (super)linear
convergence of the algorithm. Again, for the particular case of
minimizing a lower semicontinuous convex function $f$, since the subdifferential $\partial f$ is a maximal
monotone mapping (see~\cite[Theorem~A]{R70}), the linear
convergence of the algorithm is then guaranteed under metric
subregularity of the subdifferential.
Thus, the exact proximal point algorithm~\eqref{PPM_convex} is
linearly convergent when~\eqref{charac_subreg} is satisfied, and the
convergence is superlinear if $\lambda_n$ converges to $0$.

\subsection{Calmness and solution maps to parametric generalized equations}

We begin by recalling the definitions of two properties closely tied to metric subregularity: \emph{calmness} and \emph{isolated calmness}.

\begin{definition}
A set-valued mapping $F:X\tto Y$ is said to be \emph{calm}
at $\bx$ for $\by$ if  $\by\in F(\bx)$ and there is a positive constant $\kappa$ along with neighborhood $U$ of $\bx$ and $V$
of $\by$ such that
\begin{equation} \label{calm}
e(F(x)\cap V,F(\bx))\leq\kappa \|x-\bx\|, \text{ for all } x \in U.
\end{equation}
\end{definition}

\begin{definition}
A set-valued mapping $F:X\tto Y$ is said to have the \emph{isolated calmness} property
at $\bx$ for $\by$ if $F$ is calm at $\bx$ for $\by$ and, in addition, $\by$ is an isolated point of $F(\bx)$.
\end{definition}
Equivalently, $F$ has the isolated calmness property with constant $\kappa>0$ if there exist some neighborhoods $U$ of $\bx$ and $V$ of $\by$ such that
$$\|y-\by\|\leq\kappa \|x-\bx\|, \text{ for all } x \in U\text{ and }y\in F(x)\cap V.$$
It is well-known that $F$ is calm with constant $\kappa$ at some point $\bx$ for $\by$ if and only if the inverse mapping $F^{-1}$ is metrically subregular at $\by$ for $\bx$ with the same constant (see, e.g.,~\cite[Th.~3H.3]{DRO09}). Similarly, $F$ has the isolated calmness at some $\bx$ for $\by$ if and only if $F^{-1}$ is strongly metrically subregular $\by$ for $\bx$, see e.g.~\cite[Th.~3I.2]{DRO09}.

For any function $f:X\to\R\cup\{\infty\}$, the {\em Fenchel conjugate} (also called the {\em Legendre-Fenchel
conjugate} or {\em transform}) of $f$ is the function \label{note:Fenchel-conjugate}
$f^*:X^* \to [-\infty,+\infty]$ defined by
$$
f^*(x^*) := \sup_{x \in X} \{\langle x^*,x \rangle - f(x) \}.
$$
The Fenchel conjugate is always a convex and lower semicontinuous function. Moreover, if $X$ is reflexive and $f$ is proper, convex and lower semicontinuous, then $(\partial f)^{-1}=\partial f^*: X^*\tto X$. Hence we obtain the next result as a direct consequence of Theorem~\ref{th_subreg} and Theorem~\ref{th_strong_subreg}.

\begin{corollary}\label{cor:calm}
 Given a reflexive Banach space $X$, consider a function $f\in\Gamma(X)$
  and points $\bx\in X$ and $\by^*\in X^*$ such that $\by^*\in\partial f(\bx)$.
  Then the following assertions hold.
\begin{enumerate}
  \item $\partial f$ is calm at $\bx$ for
  $\by^*$ if and only if there exist a neighborhood $V$ of $\by$ and a
  positive constant $c$ such that
  \begin{equation}\label{charac_calm}
    f^*(y^*)\geq f^*(\by^*)+\langle \bx,y^*-\by^*\rangle +cd^2(y,\partial f(\bx)) \text{ whenever } y^*\in V.
  \end{equation}
  Specifically, if $\partial f$ is calm at $\bx$ for $\by^*$ with constant $\kappa$, then~\eqref{charac_calm} holds for all  $c<1/(4\kappa)$; conversely,
  if~\eqref{charac_calm} holds with constant $c$, then $\partial f$ is calm at $\bx$ for $\by^*$ with constant $1/c$.
  \item $\partial f$ has the isolated calmness property at $\bx$ for
    $\by^*$ if and only if there exist a neighborhood $V$ of $\by^*$ and a
    positive constant $c$ such that
    \begin{equation}\label{charac_isocalm}
          f^*(y^*)\geq f^*(\by^*)+\langle \bx,y^*-\by^*\rangle +c\|y^*-\by^*\|^2 \text{ whenever } y^*\in V.
    \end{equation}
    Specifically, if $\partial f$ has the isolated calmness property at $\bx$ for $\by^*$ with constant $\kappa$, then~\eqref{charac_isocalm} holds for all  $c<1/(4\kappa)$; conversely,
      if~\eqref{charac_isocalm} holds with constant $c$, then $\partial f$ has the isolated calmness property at $\bx$ for $\by^*$ with constant $1/c$.
\end{enumerate}
\end{corollary}

Our final statement provides quadratic growth characterizations of the metric subregularity and calmness properties of solution maps to parametric generalized equations.

\begin{corollary}\label{cor:2} 
Let
\begin{eqnarray}\label{sm-sub}
S(x):=\big\{y\in Y\big|\;0\in f(x,y)+\partial\phi(y)\big\},\quad x\in X,
\end{eqnarray}
define the solution map of the parametric generalized equation with the reflexive Banach space $Y$ of decision variables and the Banach space $X$ of parameters, and let $\phi\in\Gamma(Y)$. Given $(\bx,\by)\in\gph S$, assume that $f\colon X\times Y\to Y$ is Lipschitz continuous around $(\bx,\by)$ and partially strictly differentiable at this point with respect to $x$ uniformly in $y$ and that its partial derivative operator $\nabla_x f(\bx,\by)\colon X\to Y$ is surjective. The following hold:
\begin{enumerate}
\item\label{(i)} The solution map $S$ in \eqref{sm-sub} is metrically subregular at $\bx$ for $\by$ if and only if the conjugate growth condition~\eqref{charac_calm} is satisfied.

\item\label{(iv)} The solution map $S$ is strongly subregular at $\bx$ for $\by$ if and only if the conjugate growth condition~\eqref{charac_isocalm} is satisfied.

\item\label{(ii)} Suppose that the base mapping $f=f(x)$ in \eqref{sm-sub} does not depend on the decision variable $y$. Then the solution map $S$ is calm at $\bx$ for $\by$ if and only if the growth condition \eqref{charac_subreg} is satisfied.

\item\label{(iii)} If the quadratic growth condition~\eqref{charac_strong_subreg} holds with constant $c>0$ and the partial Lipschitz modulus of $f$ with respect to $y$ is smaller than $c$, then the solution map $S$ has the isolated calmness property at $\bx$ for $\by$.

\end{enumerate}
\end{corollary}
\begin{proof} Assertions~\ref{(i)} and~\ref{(iv)} follow from Corollary~\ref{cor:calm},~\cite[Cor.3.5]{AM10} and~\cite[Th.~5.10]{AM11}. Assertion~\ref{(ii)} is a consequence of Theorem~\ref{th_subreg} and  \cite[Th.~5.6]{AM10}. Assertion~\ref{(iii)} follows from Theorem~\ref{th_strong_subreg} and~\cite[Th.~4.3]{AM11}.
\end{proof}

\begin{rem}
The converse of Corolary~\ref{cor:2}\ref{(iii)} also holds under some additional conditions, see~\cite[Th.~5.5]{AM11}.
\end{rem}

\end{document}